\documentclass{amsart}
\usepackage{amssymb, amsthm, amsmath, amsfonts,amscd}
\usepackage{graphics}
\usepackage{hyperref}
\usepackage[all]{xy}
\usepackage{enumerate}
\usepackage[mathscr]{eucal}
\usepackage{bbm}

\providecommand{\U}[1]{\protect\rule{.1in}{.1in}}
\def\theenumi{\arabic{enumi}}

\def\theenumii{\alph{enumii}}
\def\p@enumii{\theenumi.}

\def\theenumiii{\arabic{enumiii}}
\def\p@enumiii{(\theenumi)(\theenumii)}

\def\p@enumiv{\p@enumiii.\theenumiii}
\theoremstyle{plain}
\newtheorem{theorem}{Theorem}[section]
\newtheorem{lemma}[theorem]{Lemma}
\newtheorem{proposition}[theorem]{Proposition}

\newtheorem{corollary}[theorem]{Corollary}
\newtheorem{conjecture}[theorem]{Conjecture}
\numberwithin{equation}{section}

\theoremstyle{definition}

\newtheorem{definition}[theorem]{Definition}

\newtheorem{remark}[theorem]{Remark}

\newtheorem{thmab}{Theorem}

\renewenvironment{proof}[1][\proofname]{{\bfseries #1\\}}{\qed}

\DeclareMathOperator{\Mod}{-Mod}
\DeclareMathOperator{\module}{-mod}
\DeclareMathOperator{\Hom}{Hom}
\DeclareMathOperator{\End}{End}
\DeclareMathOperator{\Tor}{Tor}
\DeclareMathOperator{\hd}{hd}
\DeclareMathOperator{\cd}{cd}
\DeclareMathOperator{\gd}{gd}

\DeclareMathOperator{\td}{td}
\DeclareMathOperator{\reg}{reg}
\DeclareMathOperator{\coker}{coker}
\DeclareMathOperator{\Ext}{Ext}
\DeclareMathOperator{\Res}{Res}
\DeclareMathOperator{\Ind}{Ind}
\DeclareMathOperator{\depth}{depth}
\DeclareMathOperator{\cdepth}{depth^{class}}
\newcommand{\as}{\text{*}}

\newcommand{\C}{{\mathscr{C}}}

\newcommand{\Q}{\mathbb{Q}}
\newcommand{\FI}{{\mathscr{FI}}}
\newcommand{\mi}{\mathfrak{m}}
\newcommand{\filcom}{\mathcal{C}^\bullet}
\newcommand{\fiHom}{\mathscr{H}om}
\newcommand{\fiExt}{{\mathscr{E}xt}}
\newcommand{\arXiv}[1]{\href{http://arxiv.org/abs/#1}{\nolinkurl{arXiv:#1}}}
\newcommand{\arXivV}[2]{\href{http://arxiv.org/abs/#1}{\nolinkurl{arXiv:#1v#2}}}

\title{Depth and the Local Cohomology of $\FI_G$-modules}

\author{Liping Li and Eric Ramos}
\address{College of Mathematics and Computer Science, Hunan Normal University; Key Laboratory of Performance Computing and Stochastic Information Processing (Hunan Normal University), Ministry of Education; Changsha, Hunan 410081, China.}
\email{lipingli@hunnu.edu.cn; lixxx480@umn.edu.}
\address{Department of Mathematics, University of Wisconsin - Madison.}
\email{eramos@math.wisc.edu}

\thanks{The first author would like to acknowledge the generous support provided by the National Natural Science Foundation of China 11541002, the Construct Program of the Key Discipline in Hunan Province, and the Start-Up Funds of Hunan Normal University 830122-0037. The second author was supported by NSF grant DMS-1502553.}

\begin{document}
\maketitle

\begin{abstract}
In this paper we describe a machinery for homological calculations of representations of $\FI_G$, and use it to develop a local cohomology theory over any commutative Noetherian ring. As an application, we show that the depth introduced by the second author in \cite{R} coincides with a more classical invariant from commutative algebra, and obtain upper bounds of a few important invariants of $\FI_G$-modules in terms of torsion degrees of their local cohomology groups.
\end{abstract}

\section{Introduction}

\subsection{Motivation}

Since the fundamental work of Church, Ellenberg, and Farb in \cite{CEF}, the representation theory of the category $\FI$, whose objects the finite sets $[n] = \{1,\ldots,n\}$ and whose morphisms are injections, has played a central role in Church and Farb's representation stability theory \cite{CF}. The methodology of Church, Ellenberg, and Farb has since been generalized to account for many other well known facts in asymptotic algebra. The main idea behind these generalizations is that stability properties of sequences of group representations can be converted to considerations in the representation theory of certain infinite categories, usually equipped with nice combinatorial structures. This philosophy was carried out by a series of works in this area; see \cite{CEF, CEFN, GL, N, R, SS, SS2, JW}. In this paper, we will be concerned with representations of the category $\FI_G$, where $G$ is a finite group, which was introduced in \cite{SS2}.\\

The homological aspect of the representation theory of $\FI$ and its generalization $\FI_G$ originates from the papers \cite{CEFN}, and \cite{CE} by Church, Ellenberg, Farb and Nagpal. In these papers, the notion of $\FI$-module homology was introduced, as well as the Castelnuovo-Mumford regularity (See Definitions \ref{homdef} and \ref{regdef}). The paper \cite{CE} was the first to provide explicit bounds on this regularity \cite[Theorem A]{CE}. Soon afterward, the techniques of Church and Ellenberg were expanded through two different approaches. One approach was rooted in classical representation theory, and pursued by the first author and Yu in \cite{LY, L2}. The second approach applied certain important ideas in commutative algebra, and was studied by the second author in \cite{R}. Both approaches established a strong relationship between homological properties of $\FI_G$ and its representation stability phenomena, which was accomplished by  obtaining upper bounds on certain homological invariants; see \cite[Theorem 1.3]{L2} and \cite[Theorems C and D]{R}. So far this project is still under active exploration.\\

The goals of this paper are two-fold. Firstly, it is a known fact that most finitely generated $\FI_G$-modules have infinite projective dimension (see \cite[Theorem 1.5]{LY}). Moreover, it can be shown that the category of finitely generated $\FI_G$-modules over an arbitrary Noetherian ring does not usually have sufficiently many injective objects (see Theorem \ref{noinj}). As a result of this, it becomes important for one to develop different machinery for computing certain homological invariants. The first half of this paper is concerned with creating such machinery, which will be built around the homological properties of \emph{$\sharp$-filtered} modules (see Definition \ref{sfildef}). One may think of these modules as acting as both the projective and torsion free injective objects of the category of finitely generated $\FI_G$-modules (See Theorems \ref{charfil} and \ref{homorth})\\

Secondly, we apply the homological machinery developed in the first half of the paper to the concepts of \emph{depth} (See Definition \ref{depth}) and \emph{local cohomology}. For instance, the depth and \emph{classical depth} (See Definition \ref{cdepth}) of an $\FI_G$-module were shown to be equivalent by the second author in \cite[Theorem 4.15]{R} for fields of characteristic 0. It is therefore natural to wonder whether they also coincide for arbitrary commutative Noetherian rings. Moreover, Sam and Snowden described a local cohomology theory for $\FI$ in \cite{SS3} for fields of characteristic 0, so we ask whether a generalized theory for $\FI_G$ can be developed in the much wider framework of commutative Noetherian rings. The main goal of the second half of the paper is to give affirmative answers to the above concerns.\\

\subsection{An inductive method}

An extremely useful combinatorial property of the category $\FI_G$ is that it is equipped with a self-embedding functor, which induces a shift functor $\Sigma$ in the category of $\FI_G$-modules (See definition \ref{shift}). This functor has seen heavy use in the literature. Examples of this include \cite{CEF, CEFN, CE, GL, L, LY, N, R, L2}. The importance of this functor lies in the following key facts: it preserves both left and right projective modules; and for every finitely generated $\FI_G$-module $V$, after applying the shift functor $\Sigma$ enough times, $V$ becomes a $\sharp$-filtered module \cite[Theorem A]{N}.\\

From the perspective of more classically rooted representation theory, the shift functor $\Sigma$ is a kind of restriction. It has a left adjoint functor (called \emph{induction}) and a right adjoint functor (called \emph{coinduction}), denoted by us $L$ and $R$. The coinduction functor $R$ was explicitly constructed in \cite{GL}. In this paper we systematically consider these functors, showing that their behaviors perfectly adapt to $\sharp$-filtered modules. Moreover, the Eckmann-Shapiro lemma holds in the context of $\FI_G$-modules, which allows us to reduce general questions to their simplest cases.\\

\begin{thmab}\label{bigrepthm}
Let $k$ be a commutative Noetherian ring, and let $V_n$ be a finitely generated $kG_n$-module, where $G_n = G \wr \mathfrak{S}_n$. We have:
\begin{enumerate}
\item the functors $\Sigma$, $L$, and $R$ are all exact;
\item the functors $\Sigma$, $L$, and $R$ preserve $\sharp$-filtered modules. In particular (see Definition \ref{sfildef}):
\begin{align*}
\Sigma M(V_n) & \cong M(V_n) \oplus M(\Res ^{G_n} _{G_{n-1}} V_n);\\
R(M(V_n)) & \cong M(V_n) \oplus M(\Ind _{G_n} ^{G_{n+1}} V_n);\\
L(M(V_n)) & \cong M(\Ind ^{G_{n+1}} _{G_n} V_n).\\
\end{align*}
\item For two finitely generated $\FI_G$-modules $V$ and $V'$ and $i \geqslant 0$, one has
\begin{align*}
\Ext^i_{k\FI_G}(L(V),V') & \cong \Ext^i_{k\FI_G}(V,\Sigma V');\\
\Ext^i_{k\FI_G}(\Sigma V,V') & \cong \Ext^i_{k\FI_G}(V,R(V')).
\end{align*}\\
\end{enumerate}
\end{thmab}

As an example of these induction style arguments, we prove the aforementioned equivalence of depth and classical depth. If $V$ is an $\FI_G$-module, there is always a natural map $V \rightarrow \Sigma V$ (See Definition \ref{dervdef} for the explicit definition). The cokernel of this map defines a functor, which we call the \emph{derivative} $DV$. The derivative functor was first introduced in \cite{CE}, and was later used by Yu and the authors in \cite{LY}, and \cite{R}. The paper \cite{R} used the derived functors of the derivative and its iterates, $H_i^{D^a}(\bullet)$, to develop a theory of depth for $\FI_G$-modules. It was also noted in that paper that there was a more classically rooted definition of depth called the classical depth of the module. Using Theorem \ref{bigrepthm} we will be able prove the following.\\

\begin{thmab}\label{equivdepth}
Let $V$ be a finitely generated $\FI_G$-module over a Noetherian ring $k$. Then the depth and classical depth of $V$ coincide. That is:
\begin{equation*}
\inf\{a \mid H_1^{D^{a+1}}(V) \neq 0\} = \inf\{i \mid \Ext^i_{\C\Mod}(k\C/\mi,V) \neq 0\}.
\end{equation*}
\text{}\\
\end{thmab}

\subsection{A machinery for homological calculations}

In practice, to compute the homology (resp. cohomology) groups of representations of a ring, one usually approximates these representations by suitable projective (resp. injective) resolutions. However, we have already discussed that this strategy does not work very well in the context of $\FI_G$-modules. To overcome this obstacle, one has to find suitable objects which satisfy the following two requirements: they must be acyclic with respect to certain important functors; and finitely generated $\FI_G$-modules must be approximated by finite complexes of such objects.\\

The work of Nagpal, Yu, and the authors in \cite[Theorem A]{N}, \cite[Theorem 1.3]{LY}, and \cite[Theorem B]{R}, strongly suggest that $\sharp$-filtered modules are best candidates. These objects were introduced as modules which were ``almost'' projective, and were later proven to be acyclic with respect to many natural right exact functors. One of the major realizations of this paper is that $\sharp$-filtered objects are also acyclic with respect to many left exact functors. In particular, the following results convince us that $\sharp$-filtered modules do play the role of both projective objects and injective objects for many homological computations.\\

\begin{thmab}[Homological characterizations of $\sharp$-filtered modules]\label{charfil}
Let $k$ be a commutative Noetherian ring, and let $V$ be a finitely generated $\FI_G$-module over $k$. Denote the endomorphism group of object $[s]$ in $\FI_G$ by $G_s$, $s \geqslant 0$. Then $V$ is $\sharp$-filtered if and only if it satisfies one of the following equivalent conditions:
\begin{enumerate}
\item $\Tor_i^{k\FI_G} (kG_s, V) = 0$ for $i \geqslant 1$ and $s \geqslant 0$;
\item $\Tor_1^{k\FI_G} (kG_s, V) = 0$ for $s \geqslant 0$;
\item $\Tor_i^{k\FI_G} (kG_s, V) = 0$ for $s \geqslant 0$ and a certain $i \geqslant 1$;
\item $\Ext^i_{k\FI_G} (kG_s, V) = 0$ for $i, s \geqslant 0$;
\item $\Ext^i_{k\FI_G} (T, V) = 0$ for $i \geqslant 0$ and all finitely generated torsion modules (See Definition \ref{torsion}).\\
\end{enumerate}
\end{thmab}

\begin{remark}
The first three homological characterizations of $\sharp$-filtered modules are not new. They have been described in \cite[Theorem 1.3]{LY} and \cite[Theorem B]{R}.\\

The Tor functors in the above theorem are strongly related the the notion of $\FI_G$-module homology, which we will formally define later. We state the above theorem using the language of Tor so that the relationship between the five given statements is more clear.\\
\end{remark}

\begin{thmab}[Homological orthogonal relations] \label{homorth}
Let $k$ be a commutative Noetherian ring, and let $V$ be a finitely generated $\FI_G$-module over $k$. Then
\begin{enumerate}
\item $T$ is a torsion module if and only if $\Ext _{k\FI_G}^i (T, V) = 0$ for all $\sharp$-filtered modules $V$ and all $i \geqslant 0$.
\item $V$ is an injective module if and only if $\Ext _{k\FI_G}^1 (W, V) = 0$ whenever $W$ is a $\sharp$-filtered module or $W$ is a finitely generated torsion module.
\item $\Ext _{k\FI_G}^i (V, F) = 0$ for all $\sharp$-filtered modules $F$ and all $i \geqslant 1$ if and only if $\Sigma_N V$ is a projective module for $N \gg 0$.
\item $\Ext _{k\FI_G}^i (T, V) = 0$ for $i \geqslant 1$ and all finitely generated torsion modules $T$ if and only if $V$ is a direct sum of an injective torsion module and a $\sharp$-filtered module.\\
\end{enumerate}
\end{thmab}

This should convince the reader that $\sharp$-filtered objects are acyclic with respect to many natural left and right exact functors. To prove that $\sharp$-filtered objects can be used to approximate arbitrary modules, we call upon the following theorem.\\

\begin{theorem}[\cite{L2}, Theorem 1.3] \label{filcombounds}
Let $k$ be a commutative Noetherian ring, and let $V$ be a finitely generated $\FI_G$-module over $k$. Then there is a complex
\[
\mathcal{C^\bullet}V : 0 \rightarrow V \rightarrow F^0 \rightarrow \ldots \rightarrow F^{n} \rightarrow 0
\]
enjoying the following properties:
\begin{enumerate}
\item every $F^i$ is a $\sharp$-filtered module;
\item $\gd(F^i) \leq \gd(V) - i$ (see Definition \ref{regdef}). Therefore, $n \leq \gd(V)$;
\item the cohomologies $H^i(\mathcal{C}^\bullet V)$ are finitely supported (See Definition \ref{regdef}), and
\[
\begin{cases} \td(H^{i}(V)) = \td(V) &\text{ if $i = -1$}\\ \td(H^{i}(V)) \leq 2\gd(V) - 2i  - 2 &\text{ if $0 \leq i \leq n$.}\end{cases}
\]\\
\end{enumerate}
\end{theorem}

We will find this theorem vital in our studies of the local cohomology of $\FI_G$-modules.\\

\subsection{A local cohomology theory}

Another application of the above machinery and inductive method is in developing a local cohomology theory of $\FI_G$-modules over arbitrary commutative Noetherian rings, generalizing the corresponding work in \cite{SS3}.\\

Given a finitely generated $\FI_G$-module $V$, there is a natural exact sequence
\begin{equation*}
0 \to V_T \to V \to V_F \to 0
\end{equation*}
where $V_T$ is torsion and $V_F$ is torsion free (see Definition \ref{regdef}). The assignment $V \to V_T$ gives rise to a functor, which we denote $H^0_{\mathfrak{m}} (\bullet)$. This is a left exact functor, and its right derived functors $H^i_{\mathfrak{m}} (\bullet)$ are called the \emph{local cohomology functors}. Considering the profound impact that local cohomology has in more classical settings, it is natural for one to wonder whether the same is true for $\FI_G$-modules.\\

Before one considers applying local cohomology modules in bounding various homological invariants, it is important that we develop some way of computing the modules in a systematic way. Using the above homological computation machinery and inductive method, we can provide such a computational tool.\\

\begin{thmab}\label{coequiv}
Let $k$ be a commutative Noetherian ring, $V$ be a finitely generated $\FI_G$-module, and let $\filcom V$ be the complex in Theorem \ref{filcombounds}. Then, for $i \geqslant -1$, there are isomorphisms
\begin{equation*}
H^i(\filcom V) \cong H^{i+1}_\mi(V).
\end{equation*}
Consequently, $H^i_\mi(V)$ is a finitely generated, torsion $\C$-module. Moreover,
\begin{equation*}
\begin{cases}
\td(H^i_\mi(V)) = \td(V) &\text{ if $i = 0$}\\
\td(H^i_\mi(V)) \leq 2\gd(V) - 2i - 2 & \text{ if $1 \leq i \leq \gd(V)$.}
\end{cases}
\end{equation*}\\
\end{thmab}

Using this theorem, we can show that local cohomology groups are related to certain important homological invariants such as the depth, \emph{Nagpal number} (see Definition \ref{nagnum}), and \emph{regularity} of a module (see Definition \ref{regdef}).

\begin{thmab}\label{localcohomologybounds}
Let $k$ be a commutative Noetherian ring, and let $V$ be a finitely generated $\FI_G$-module. Then:
\begin{enumerate}
\item the depth of $V$ is the smallest integer $i$ such that $H^i_\mi(V) \neq 0$;
\item the Nagpal number of $V$, $N(V)$ satisfies the bounds
\begin{equation*}
N(V) = \max\{\td(H^i_\mi(V))\mid i \geqslant 0\} + 1 \leq \max\{\td(V),2\gd(V) - 2\} + 1.
\end{equation*}
\item The regularity of $V$, $\reg(V)$ satisfies the bounds
\begin{equation*}
\reg(V) \leqslant \max\{\td(H^i_\mi(V)) + i\} \leqslant \max\{2\gd(V ) - 1, \td(V )\}.
\end{equation*}\\
\end{enumerate}
\end{thmab}

Motivated by the classical result for polynomial rings, we also state the following conjecture.\\

\begin{conjecture}
Let $k$ be a commutative Noetherian ring, and let $V$ be a finitely generated $\FI_G$-module. Then
\begin{equation*}
\reg(V) = \max\{\td(H^i_\mi(V)) + i\}.
\end{equation*}\\
\end{conjecture}

The previous theorem tells us that the right hand side of our conjectured identity is an upper bound on the left hand side. It therefore only remains to show the opposite inequality. We note that the natural way one might try to accomplish this, namely through some induction argument on the projective dimension, cannot work in this case for reasons we have already stated. If this conjecture proves to be false, it would provide evidence that the Hilbert Syzygy theorem is actually vital to the statement being true in the classical case.\\

\subsection{Organization}

This paper is organized as follows. In Section 2 we describe some elementary definitions and results, which will be used throughout this paper. In Section 3 we study the shift functor and its adjoint functors, prove a few crucial technical tools, and use them to show that the depth and classical depth coincide. In Section 4 we use the inductive tools developed in Section 4 to prove a variety of homological structure theorems about $\sharp$-filtered modules. Finally, in the last section we develop a local cohomology theory of $\FI_G$-modules and discuss its applications. We also present the above conjecture and its useful consequences.

\section{Preliminaries}

\subsection{Elementary Definitions}

For the remainder of this paper we fix a finite group $G$, and a commutative Noetherian ring $k$. The category $\FI_G$ is that whose objects are the finite sets $[n] := \{1,\ldots,n\}$ and whose morphisms are pairs, $(f,g):[n] \rightarrow [m]$, of an injection $f:[n] \hookrightarrow [m]$ with a map of sets $g:[n] \rightarrow G$. Given two composable morphisms $(f,g),(f',g')$, composition in this category is defined by
\[
(f,g) \circ (f',g') = (f \circ f', h)
\]
where $h(x) = g'(x)g(f'(x))$. It follows immediately from this that for any $[n]$, the endomorphisms of $[n]$ form the group $\End_{\FI_G}([n]) = G \wr \mathfrak{S}_n$, where $\mathfrak{S}_n$ is the symmetric group on $n$ letters. We will write $G_n$ as a shorthand for this group.\\

Two important special cases of $\FI_G$ are those wherein $G = 1$ is the trivial group, and $G = \mathbb{Z}/2\mathbb{Z}$. In the first case, $\FI_G$ is equivalent to the category $\FI$ of finite sets and injective morphisms. In the second case we have $\FI_G = \FI_{BC}$, which was studied in \cite{JW}.\\

\begin{definition}
An \textbf{$\FI_G$-module} is a covariant functor $V$ from $\FI_G$ to the category of $k$-modules. We write $V_n$ for the module $V([n])$, and, given any map $(f,g):[n] \rightarrow [m]$, we write $(f,g)_\as$ for $V(f,g)$. We call the morphisms $(f,g)_\as$ the \textbf{induced maps} of $V$. In the specific case where $n < m$, we call $(f,g)_\as$ a \textbf{transition map} of $V$.\\

The collection of $\FI_G$-modules, with natural transformations, form an abelian category which we denote $\FI_G\Mod$.\\
\end{definition}

The definition for $\FI_G$-module given above was introduced by Church, Ellenberg, and Farb in \cite{CEF}. This was followed by the work of Wilson \cite{JW}, as well as that of Sam and Snowden \cite{SS}\cite{SS2}. More recently, a new approach to the subject has been considered, which is more rooted in classical representation theory. This can be seen in the works of Gan, the first author, and Yu \cite{GL} \cite{L} \cite{LY}\\

\begin{definition}
Let $k\FI_G$ denote the \textbf{category algebra} whose additive group is given by
\[
k\FI_G := \bigoplus_{n \leq m} k[\Hom_{\FI_G}([n],[m])],
\]
where $k[\Hom_{\FI_G}([n],[m])]$ is the free $k$-module with a basis indexed by the set $\Hom_{\FI_G}([n],[m])$. Multiplication in $k\FI_G$ is defined on basis vectors $(f,g):[n] \rightarrow [m], (f',g'):[r] \rightarrow [s]$ by
\[
(f,g) \cdot (f',g') = \begin{cases} (f,g) \circ (f',g') & \text{ if $n = s$}\\ 0 & \text{ otherwise.}\end{cases}
\]
Write $e_n \in \End_{\FI_G}([n])$ for the identity on $[n]$ paired with the trivial map into $G$. Then we say that a module $V$ over $k\FI_G$ is \textbf{graded} if $V = \bigoplus_n e_n\cdot V$. In this case we write $V_n := e_n \cdot V$.\\
\end{definition}

\begin{remark}
Because all $k\FI_G$-modules considered in this paper are graded, we will simply refer to them as $k\FI_G$-modules.\\
\end{remark}

If $V$ is a $k\FI_G$-module, then we obtain an $\FI_G$-module by setting $V_n := e_n \cdot V$, and defining the induced maps in the obvious way. It is clear that this defines an equivalence between the category of $\FI_G$-modules, and the category of (graded) $k\FI_G$-modules. We will use both definitions interchangeably during the course of this paper.\\

\begin{remark}
In everything that follows, our results will not depend on the finite group $G$. Therefore, to clarify notation, we will write $\C := \FI_G$.\\
\end{remark}

\begin{definition}\label{sfildef}
Let $W$ be a $kG_n$-module for some $n \geqslant 0$. Then the \textbf{basic filtered} module over $W$, is the $\C$-module $M(W)$ defined by the assignments
\[
M(W)_m = k[\Hom_{\C}([n],[m])] \otimes_{kG_n} W.
\]
The induced maps of $M(W)$ are defined by composition on the first coordinate. In the special case where $W = kG_n$, we write $M(n) := M(W)$, and refer to direct sums of these modules as being \textbf{free}.\\

Since $kG_n$ can be viewed as a subalgebra of $k\C$, we see that $M$ is nothing but the induction functor $k\C \otimes _{kG_n} \bullet$.\\

We say that a $\C$-module $V$ is \textbf{$\sharp$-filtered} if it admits a filtration
\[
0 = V^{(-1)} \subseteq V^{(0)} \subseteq V^{(1)} \subseteq \ldots \subseteq V^{(n)} = V
\]
such that $V^{(i)}/V^{(i-1)} = M(W^{(i)})$, for some $kG_i$-module $W^{(i)}$, for each $i \geqslant 0$.\\
\end{definition}

It was shown in \cite{LY}, as well as \cite{R}, that $\sharp$-filtered objects are of a fundamental importance to the study of homological properties of $\C$-modules. For instance, $\sharp$-filtered objects are precisely the acyclic objects with respect to certain natural right exact functors. These will be discussed in the coming sections (See Theorems \ref{homacyclic} and \ref{depthclass}). One of the interesting consequences of the results in this paper is that $\sharp$-filtered objects are also acyclic with respect to certain left exact functors.\\

The following proposition follows easily from the relevant definitions.\\

\begin{proposition}\label{yoneda}
Let $W$ be a $kG_n$-module. Then there is a natural adjunction,
\[
\Hom_{\C\Mod}(M(W),V) \cong \Hom_{kG_n\Mod}(W,V_n),
\]
given by
\[
\phi \mapsto \phi_n.tor
\]
\end{proposition}

One observes that for a $kG_n$-module $W$, $M(W)$ is projective if and only if $W$ is projective. In fact, it can be shown that all projective $\C$-modules are direct sums of basic filtered modules \cite[Proposition 2.13]{R}. Therefore, free $\C$-modules are always projective.\\

\begin{definition}
Given a finitely generated $kG_n$-module $W$, we say that $M(W)$ is \textbf{finitely generated}. We say that a $\sharp$-filtered object is \textbf{finitely generated} if the cofactors in its defining filtration are finitely generated. Finally, we say that a $\C$-module is finitely generated if and only if it is a quotient of a finitely generated $\sharp$-filtered object.\\

We denote the category of finitely generated $\C$-modules by $\C\module$.
\end{definition}

\begin{remark}
One immediately remarks that the free module $M(n)$ is finitely generated for all $n$. In fact, it is an easily seen consequence of Nakayama's Lemma (Lemma \ref{nak}) that a module $V$ is finitely generated if and only if there is a finite set $I$, and a finite collection of non-negative integers $\{n_i,m_i\}_{i \in I}$ such that $V$ is a quotient of $\bigoplus_{i \in I} M(n_i)^{m_i}$.\\

Proposition \ref{yoneda} informs us that a map $M(n) \rightarrow V$ is equivalent to a choice of an element of $V_n$. Putting everything together, we can conclude that $V$ is finitely generated if and only if there is some finite set of elements in $\sqcup_{n} V_n$ which are not contained in any proper submodule. We call such a set a \textbf{generating set of elements} for $V$. While a definition of this sort may seem more natural, we use the above definition to be more in line with the philosophy of Theorem \ref{homacyclic}.\\
\end{remark}

One remarkable fact about the category $\C\module$ is that it is also abelian. Given any morphism of finitely generated $\C$-modules, it is clear that its image and cokernel must also be finitely generated. The significance of saying that $\C\module$ is abelian therefore comes from the fact that the kernel of this morphism must also be finitely generated. Put another way, the category $\C\module$ is Noetherian.\\

This fact was proven by Sam and Snowden in \cite[Corollary 1.2.2]{SS2} for general $\C$-modules, although it had been proven in certain specific cases earlier \cite[Theorem 4.21]{JW}\cite[Theorem 1.3]{CEF}\cite[Theorem A]{CEFN}\cite[Theorem 2.3]{S}. We often refer to the following result as the \textbf{Noetherian property}.\\

\begin{theorem}[\cite{SS2}, Corollary 1.2.2]
If $V$ is a finitely generated $\C$-module over a Noetherian ring $k$, then all submodules of $V$ are also finitely generated.\\
\end{theorem}

Much of the remainder of the paper will be concerned with various homological properties of the category $\C\module$.\\

\subsection{The Homology Functors}

For the remainder of this paper, we write $\mi \subseteq k\C$ to denote the ideal
\[
\mi := \bigoplus_{n < m} k[\Hom_{\C}([n],[m])].
\]

\begin{definition}\label{homdef}
For any $k\C$-module $V$, we use $H_0(V)$ to denote $k\C/\mi \otimes _{k\C} V$. In the language of $\C$-modules, $H_0:\C\module \rightarrow \C\module$ is the functor defined by
\[
H_0(V)_n = V_n/V_{<n}
\]
where $V_{<n}$ is the submodule of $V_n$ spanned by the images of transition maps originating from $V_m$ with $m < n$. We use $H_i$ to denote the derived functors
\[
H_i(V) := \Tor_i ^{k\C} (k\C/\mi, V)
\]
We call the functors $H_i$ the \textbf{homology functors}.\\
\end{definition}

\begin{proposition}[Nakayama's Lemma]\label{nak}
Let $V$ be a $\C$-module, let $\{\widetilde{v}_i\} \subseteq \sqcup_n H_0(V)_n$ be a collection of elements which generate $H_0(V)$, and let $v_i$ be a lift of $\widetilde{v}_i$ for each $i$. Then $\{v_i\}$ is a generating set for $V$.\\
\end{proposition}

\begin{proof}
Let $j$ be the least index such that $V_j \neq 0$. Then $V_j = H_0(V)_j$, and it is clear that every element of $V_j$ is a linear combination of those $\widetilde{v}_i$ in $H_0(V)_j$. Next, let $n > j$, and let $v \in V_n$. Then the image of $v$ in $H_0(V)_n$ can be expressed as some linear combination of elements of $\{\widetilde{v}_i\}$. By definition, this implies that $v$ is a linear combination of elements of $\{v_i\}$, as well as images of elements from lower degrees. The result now follows by induction.\\
\end{proof}

One immediate consequence of Nakayama's Lemma is that if $V$ is finitely generated, then $H_0(V)$ is supported in finitely many degrees.\\

\begin{definition}\label{regdef}
Given a $\C$-module $V$, we define its \textbf{support} to be the smallest integer $N$ for which $V_n = 0$ for all $n > N$, if such an integer exists. We define the \textbf{torsion degree} of $V$ by
\[
\td(V) := \sup\{n \mid \Hom_{k\C}(kG_n,V) \neq 0\} \in \mathbb{N} \cup \{-\infty, \infty\},
\]
where $\td(V) = -\infty$ if and only if $V$ is \textbf{torsion free}. Note that if $V$ is finitely supported, then $\td(V)$ is precisely its support.\\

The \textbf{$i$-th homological degree} of $V$ is defined to be the quantity,
\[
\hd_i(V) := \td(H_i(V)).
\]
The zeroth homological degree is often referred to as the \textbf{generating degree} of $V$ and written $\gd(V)$.\\

The \textbf{regularity} of $V$, denoted $\reg(V)$, is the smallest integer $N$ such that
\[
\hd_i(V) \leq N + i
\]
for all $i \geqslant 1$. We say that $\reg(V) = \infty$ if no such $N$ exists, and we say $\reg(V) = -\infty$ if $V$ is acyclic with respect to the homology functors.\\
\end{definition}

\begin{remark}
It is a subtle but important point that the regularity of a module is defined by only bounding the higher homologies. From the perspective of classical commutative algebra this might seem a bit strange. We stress, however, that Theorem \ref{localcohomologybounds} is false if the definition of regularity is altered to include $\td(H_0(V))$. From the perspective of local cohomology this can be explained by the fact that $\sharp$-filtered objects are local cohomology acyclic, despite also being those objects which one might consider as substitutes for projectives (see Corollary \ref{filacyclic}).\\
\end{remark}

It is a remarkable fact that the regularity of any finitely generated $\C$-module is not $\infty$. This was proven in the case of $\FI$-modules in characteristic 0 by Sam and Snowden \cite[Corollary 6.3.5]{SS3}, general $\FI$-modules by Church and Ellenberg \cite[Theorem A]{CE}, and for $\C$-modules by the authors and Yu \cite[Theorem A]{R} \cite[Theorem 1.8]{LY}. Most of these papers also provide explicit bounds on the regularity of a $\C$-module in terms of its first two homological degrees. Later, we will provide new bounds on the regularity of a module in terms of its local cohomology modules.\\

As previously stated, $\sharp$-filtered modules are precisely those which are acyclic with respect to the homology functors.\\

\begin{theorem}[\cite{LY} Theorem 1.3, \cite{R} Theorem B] \label{homacyclic}
Let $V$ be a finitely generated $\C$-module. Then the following are equivalent:
\begin{enumerate}
\item $V$ is $\sharp$-filtered;
\item $V$ is homology acyclic;
\item $H_1(V) = 0$;
\item $H_i(V) = 0$ for some $i \geqslant 1$.\\
\end{enumerate}
\end{theorem}

\begin{remark}
Note that this theorem implies the first three conditions of Theorem \ref{charfil}.\\
\end{remark}

It follows as a consequence of this theorem that the only finitely generated modules which can have finite projective dimension are $\sharp$-filtered objects. The question of what $\sharp$-filtered modules can have finite projective dimension is considered in \cite[Theorem 1.5]{LY}.\\

\section{The Shift Functor and its Adjoints}

\subsection{The shift functor}

\begin{definition}\label{shift}
For any morphism $(f,g):[n] \rightarrow [m]$ in $\C$, we define $(f_+,g_+):[n+1] \rightarrow [m+1]$ to be the morphism
\[
f_+(x) = \begin{cases} f(x) &\text{ if $x < n+1$}\\ m+1 &\text{ otherwise}\end{cases}, \text{  } g_+(x) = \begin{cases} g(x) &\text{ if $x < n+1$}\\ 1 &\text{ otherwise.}\end{cases}
\]
Let $\iota:\C \rightarrow \C$ be the functor defined by
\[
\iota([n]) = [n+1], \text{  } \iota(f,g) = (f_+,g_+).
\]
Then we define the \textbf{shift functor}, or the \textbf{restriction functor}, $\Sigma:\C\Mod \rightarrow \C\Mod$ by $\Sigma V := V \circ \iota$. We write $\Sigma_b$ for the $b$-th iterate of $\Sigma$.\\

We note that the map $\iota$ induces a proper injective map of algebras $k\C \rightarrow k\C$, which we call the \textbf{self-embedding} of $k\C$.\\
\end{definition}

One of the most important properties of the shift functor is that it preserves $\sharp$-filtered objects. This was first observed by Nagpal in \cite[Lemma 2.2]{N}.\\

\begin{proposition} \label{genetic}
Let $W$ be a $kG_n$-module. Then there is an isomorphism of $\C$-modules
\[
\Sigma M(W) \cong M(\Res_{G_{n-1}}^{G_n} W) \oplus M(W).
\]
In particular, if $X$ is a $\sharp$-filtered $\C$-module, then $\Sigma V$ is as well.\\
\end{proposition}

\begin{proof}
We will construct the isomorphism here, and direct the reader to \cite[Proposition 2.21]{R} for a proof that the map we construct is actually an isomorphism.\\

Proposition \ref{yoneda} implies that any map $M(\Res_{G_{n-1}}^{G_n} W) \oplus M(W) \rightarrow \Sigma M(W)$ is determined by two maps, $\phi_1:\Res_{G_{n-1}}^{G_n} W \rightarrow \Sigma M(W)_{n-1}$ and $\phi_2:W \rightarrow \Sigma M(W)_n$. It is clear from definition that $M(W)_{n-1} = \Res_{G_{n-1}}^{G_n} W$, and so we may choose $\phi_1$ to be the identity. Once again applying the definition of the shift functor, we note that the only pure tensors $w \otimes (f,g) \in \Sigma M(W)_n$ which are not in the image of a transition map are those for which $f^{-1}(n+1) = \emptyset$. Looking at the collection of all such pure tensors, we find that they form a copy of $W$ in $\Sigma M(W)_n$. We define
\[
\phi_2(w) = w \otimes (f_n,\mathbf{1})
\]
where $f_n:[n] \rightarrow [n+1]$ is the standard inclusion, and $\mathbf{1}$ is the trivial map into $G$.\\
\end{proof}

One remarkable fact about the shift functor, first observed by Nagpal in \cite[Theorem A]{N}, is that all finitely generated $\C$-modules are ``eventually'' $\sharp$-filtered.\\

\begin{theorem}[\cite{N}, Theorem A]
Let $V$ be a finitely generated $\C$-module. Then $\Sigma_b V$ is $\sharp$-filtered for $b \gg 0$.\\
\end{theorem}

One major consequence of this theorem is the phenomenon of the stable range. If $k$ is a field, and $W$ is a finite dimensional $kG_n$-module, then one may easily compute
\[
\dim_k(M(W)_m) = \binom{m}{n}\dim_k W.
\]
Nagpal's theorem therefore implies the following.\\

\begin{corollary}
Let $V$ be a finitely generated $\C$-module over a field $k$. Then there is a polynomial $P_V \in \Q[x]$ such that for all $n \gg 0$, $\dim_k V_n = P_V(n)$.\\
\end{corollary}

\begin{definition}\label{nagnum}
Let $V$ be a finitely generated $\C$-module. The smallest $b$ for which $\Sigma_b V$ is $\sharp$-filtered is known as the \textbf{Nagpal number} of $V$, and is denoted $N(V)$.\\
\end{definition}

The paper \cite[Theorem C]{R} examines the Nagpal number from the perspective of a theory of depth. In this work, the second author provides bounds on $N(V)$ in terms of the generating degree and the first homological degree. Similar bounds were later found by the first author in \cite[Theorem 1.3]{L2} using different means. We will examine the notion of depth in the coming sections, and its connection to local cohomology.\\

\subsection{The coinduction functor}

\begin{definition}
If $V$ is a $\C$-module, then we define the \textbf{coinduction functor} $R:\C\Mod \rightarrow \C\Mod$
\[
R(V)_n := \Hom_{k\C}(\Sigma M(n),V).
\]
If $(f,g):[n] \rightarrow [m]$ is a morphism in $\C$, and $\phi: \Sigma M(n) \rightarrow V$ is a morphism of $\C$-modules, then we define $(f,g)_\as \phi \in \Hom_{k\C}(\Sigma M(m),V)$ by
\[
((f,g)_\as \phi)_r(f',g') = \phi_r( (f',g') \circ (f,g))
\]
where $(f',g'):[m] \rightarrow [r+1]$.\\
\end{definition}

\begin{remark}
Because it will be useful to us later, we note that the coinduction functor is exact. Indeed, we have already seen that the shift functor preserves projective objects, and that $M(n)$ is projective for all $n$. This implies that $\Hom_{k\C}(\Sigma M(n), \bullet)$ is exact for all $n$, whence $R$ is exact.\\
\end{remark}

\begin{proposition}[\cite{GL}, Lemma 4.2]
The coinduction functor is right adjoint to the shift functor.\\
\end{proposition}

The coinduction functor was introduced by Gan and the first author in \cite{GL}. Theorem \ref{coind} generalizes Theorem 1.3 of that paper. We will eventually use this more general result to prove that depth and classical depth agree for $\C$-modules over any Noetherian ring (see Theorem \ref{equivdepth}).\\

\begin{lemma}\label{coindlem1}
Let $W$ be a finitely generated $kG_n$-module. Then $M(W)$ is a summand of $R(M(W))$.
\end{lemma}

\begin{proof}
We will construct a split injection $M(W) \rightarrow R(M(W))$. Proposition \ref{yoneda} tells us that such a map is equivalent to a map $W \rightarrow R(M(W))_n = \Hom_{k\C}(\Sigma M(n),M(W))$. Proposition \ref{genetic} tells us that
\[
\Sigma M(n) \cong M(n-1)^{|G|\cdot n} \oplus M(n),
\]
and therefore
\[
\Hom_{k\C}(\Sigma M(n),M(W)) \cong \Hom_{k\C}(M(n),M(W)) \cong W.
\]
Being explicit, the isomorphism $\Hom_{k\C}(\Sigma M(n),M(W))  \cong W$ is given by
\[
\phi \mapsto \phi_n(f_n,\mathbf{1})
\]
where $f_n:[n] \rightarrow [n+1]$ is the standard inclusion, and $\mathbf{1}$ is the trivial map into $G$. We claim that the map $M(W) \rightarrow R(M(W))$ induced by the identity on $W$ is a split injection.\\

To prove the claim, we will construct a splitting $\psi:R(M(W)) \rightarrow M(W)$. Indeed, for a morphism $\phi:\Sigma M(r) \rightarrow M(W)$ we set
\[
\psi_r(\phi) = \phi_r(f_r,\mathbf{1}),
\]
where $f_r:[r] \rightarrow [r+1]$ is the standard inclusion, and $\mathbf{1}$ is the trivial map into $G$. The fact that $\psi$ defines a morphism of $\C$-modules is routinely checked. In fact, if $(f,g):[r]\rightarrow [s]$ is any map in $\C$ then,
\[
(f,g)_\as (\psi_r(\phi)) = (f,g)_\as \phi_r(f_r,\mathbf{1}) = \phi_s( (f_+,g_+) \circ (f_r,\mathbf{1})) = \phi_s( (f_s,\mathbf{1}) \circ (f,g)) = ((f,g)_\as\phi)_s(f_s,\mathbf{1}) = \psi_s((f,g)_\as\phi)
\]
By the discussion in the previous paragraph, it is clear that $\psi$ is a splitting of our map.\\
\end{proof}

This is all we need to prove the main theorem of this section.\\

\begin{theorem}\label{coind}
Let $W$ be a finitely generated $kG_n$-module. Then $R(M(W))$ is a direct sum of basic filtered modules. More specifically,
\[
R(M(W)) \cong M(W) \oplus M(\Ind_{G_n}^{G_{n+1}}W).
\]
\end{theorem}

\begin{proof}
We may find some integer $m$ such that there is an exact sequence
\[
0 \rightarrow M(Z) \rightarrow M(n)^m \rightarrow M(W) \rightarrow 0
\]
where $Z$ is some $kG_n$-module. Applying the exact coinduction functor, and using \cite[Theorem 1.3]{GL}, we obtain an exact sequence
\[
0 \rightarrow R(M(Z)) \rightarrow M(n)^m \oplus M(n+1)^m \rightarrow R(M(W)) \rightarrow 0
\]
It follows from this that $\hd_1(R(M(W))) \leq n+1$, and that $R(M(W))$ is generated in degrees $n$ and $n+1$.\\

Lemma \ref{coindlem1} tells us that the submodule of $R(M(W))$ generated by $R(M(W))_n \cong W$ is precisely $M(W)$. That is, there is a split exact sequence
\[
0 \rightarrow M(W) \rightarrow R(M(W)) \rightarrow Q \rightarrow 0
\]
for some module $Q$ generated in degree exactly $n+1$. Applying the $H_0$ functor to this sequence, we find that
\[
H_1(R(M(W))) \rightarrow H_1(Q) \rightarrow H_0(M(W))
\]
from which it follows that $\hd_1(Q) \leq n+1$. The argument of \cite[Corollary 3.4]{LY} now implies that $Q$ is actually a basic filtered module $Q \cong M(U)$ for some $kG_{n+1}$-module $U$.\\

Now we complete the proof by showing that $U \cong \Ind_{G_n}^{G_{n+1}}W$. By considering the value of $R(M(W))$ on the object $n+1$, we have
\begin{align*}
R((M(W))_{n+1} &= \Hom_{k\C} (\Sigma(M(n+1)), M(W))\\
& \cong \Hom_{k\C} (M(n+1) \oplus M(n)^{(n+1)|G|}, M(W))\\
& \cong M(W)_{n+1} \oplus M(W)_n^{(n+1)|G|}
\end{align*}
The endomorphism group $G_{n+1}$ acts transitively on the $(n+1)|G|$ copies of $M(W)_n \cong W$. Therefore, as a left $kG_{n+1}$-module, $R(W)_{n+1}$ is a direct sum of $W_{n+1}$ and $\Ind _{G_n}^{G_{n+1}} W_n$. Note that these two direct summands are actually isomorphic since
\begin{equation*}
M(W)_{n+1} \cong k[\Hom_{\C}(n, n+1)] \otimes _{kG_n} W \cong kG_{n+1} \otimes _{kG_n} W,
\end{equation*}
where $k[\Hom_{\C}(n, n+1)]$ as a $(kG_{n+1}, kG_n)$-bimodule is isomorphic to $kG_{n+1}$. But we have
\begin{equation*}
R(W)_{n+1} \cong W_{n+1} \oplus U.
\end{equation*}
Since we already know that $R(W)_{n+1}$ is a direct sum of two copies of the induced module, it forces $U$ to be isomorphic to the induced module.\\
\end{proof}

\subsection{The induction functor and the proof of Theorem \ref{bigrepthm}}

We now spend some time considering the left adjoint of the shift functor. For our purposes, its most important property will be related to the Eckmann-Shapiro Lemma (Proposition \ref{EckSha}).\\

Unlike the coinduction functor, we will see that the shift functor's left adjoint cannot be easily expressed in the language of $\C$-modules. We will therefore present this functor entirely in the language of $k\C$-modules.\\

\begin{definition}
Let $V$ be a $k\C$-module. Set $\C_+$ to be the full subcategory of $\C$ whose objects are the sets $[n]$ with $n > 0$, and let $k\C_+$ be the corresponding subalgebra of $k\C$. Then we define the \textbf{Induction functor} $L$ as the $k\C$-module,
\[
L(V) := k\C_+ \otimes_{k\C} V
\]
where here we consider $k\C_+$ as a $k\C$-bimodule via normal multiplication on the left, and via the self-embedding on the right.\\
\end{definition}

\begin{proposition}\label{leftadjoint}
The induction functor is left adjoint to the shift functor.\\
\end{proposition}

\begin{proof}
Let $k\C_+$ be as in the definition of the induction functor. We will prove that there is a natural isomorphism of functors
\[
\Sigma V \cong \Hom_{k\C}(k\C_+,V).
\]
Keeping this in mind, the proposition is just the usual Tensor-Hom adjunction.\\

For this proof only, set $A := \Hom_{k\C}(k\C_+,V)$. Then $A$ is a $k\C$-module. Recall that we use $e_n$ to denote the morphism of $\C$ defined by the pair of the identity on $[n]$ and the trivial map into $G$. For any $\phi \in A_n$, define
\[
\psi_V(\phi) = \phi(e_{n+1}) \in V_{n+1}.
\]
We claim that $\psi:A \rightarrow \Sigma V$ is a morphism of $k\C$-modules. Indeed, if $(f,g):[n] \rightarrow [m]$ is any morphism in $\C$, and $\phi \in A_n$, then
\[
\psi_V((f,g)_\as \phi) = ((f,g)_\as\phi)(e_{m+1}) = \phi(e_{m+1} \circ (f_+,g_+)) = (f_+,g_+)_\as \phi(e_{n+1}) = (f,g)_\as (\psi_V(\phi)).
\]
The fact that the collection of $\psi_V$, with $V$ varying, define a natural transformation of functors is easily checked.\\
\end{proof}

\begin{remark}
The induction and coinduction functors of $\C$-modules are not isomorphic. Indeed, we have already seen that $R(M(n)) \cong M(n) \oplus M(n+1)$, while Theorem \ref{bigrepthm} tells us that $L(M(n)) \cong M(n+1)$.\\
\end{remark}

We have already seen that the shift functor preserves projective $\C$-modules (Proposition \ref{genetic}). For the purposes of the Eckmann-Shapiro lemma, however, we will need to know whether it preserves right projective modules. This is indeed the case.\\

\begin{lemma}\label{rightproj}
Let $P$ be a projective right $k\C$-module. Then $\Sigma P$ is also a projective right $k\C$-module.\\
\end{lemma}

\begin{proof}
It suffices to show that the projective modules of the form $e_i \cdot k\C$ remain projective once shifted. In fact, it is the case that
\[
\Sigma(e_n\cdot k\C) \cong (e_{n-1}\cdot k\C)^{n}.
\]
This is easily checked by explicitly computing the action.\\
\end{proof}

Putting everything in the previous two sections together, we can finally prove Theorem \ref{bigrepthm}.\\

\begin{proof}[The proof of Theorem \ref{bigrepthm}]
We have already seen that $\Sigma$ and $R$ are exact. The fact that $L$ is also exact follows immediately from its definition. Indeed, it is a simple computation to show that
\[
L(V)_{n+1} = \Ind_{G_n}^{G_{n+1}} V_n
\]
for any $n \geq 0$. Because kernels and cokernels are computed pointwise, exactness of $L$ follows from the exactness of the classical induction functor.\\

Proposition \ref{genetic}, and Theorem \ref{coind} imply the first two claims of the second statement. Let $W$ be a $kG_n$-module. Then by Proposition \ref{leftadjoint}, if $V$ is a $\C$-module
\[
\Hom_{k\C}(L(M(W)), V) \cong \Hom_{k\C}(M(W),\Sigma V) \cong \Hom_{kG_n}(W,\Res_{G_n}^{G_{n+1}}V_{n+1})
\cong \Hom_{kG_{n+1}}(\Ind_{G_n}^{G_{n+1}} W, V_{n+1}).
\]
Proposition \ref{yoneda} implies that $L(M(W)) \cong M(\Ind_{G_n}^{G_{n+1}} W)$, as desired.\\

The final part of the theorem is just the Eckmann-Shapiro lemma from representation theory, and is proven in the same way in this context.\\
\end{proof}

For ease of reference later, we take a moment to state the Eckmann-Shapiro lemma.\\

\begin{proposition}[The Eckmann-Shapiro lemma]\label{EckSha}
Let $V$ and $V'$ be finitely generated $\C$-modules. Then there are isomorphisms for all $i \geqslant 0$
\begin{eqnarray}
\Ext^i_{k\C}(L(V),V') &\cong& \Ext^i_{k\C}(V,\Sigma V');\\
\Ext^i_{k\C}(\Sigma V,V') &\cong& \Ext^i_{k\C}(V,R(V')).
\end{eqnarray}\\
\end{proposition}

\section{Homological computations}

\subsection{Depth and the proof of Theorem \ref{equivdepth}}

In this section we will review the concept of depth first introduced in \cite{R}. Following this, we will consider an invariant we call classical depth, and prove that it is equivalent to the depth of \cite{R}.

\begin{definition}\label{dervdef}
Let $V$ be a $\C$-module, and let $\tau_V:V \rightarrow \Sigma V$ be the map induced by the pairs $(f_n,\mathbf{1})$, where $f_n:[n] \rightarrow [n+1]$ is the standard inclusion and the trivial map into $G$. We define the \textbf{derivative functor} to be the cokernel
\[
DV := \coker(\tau_V).
\]
We write $D^a$ for the $a$-th iterate of $D$.\\

Note that the derivative functor is right exact, and so we write $H_i^{D^a}$ to denote the $i$-th derived functor of $D^a$. By \cite[Lemma 3.6]{CE}, $V$ is torsion free if and only if $\tau_V$ is injective.\\
\end{definition}

\begin{remark}\label{dervprop}
It follows from the proof of Proposition \ref{genetic} that $\tau_{M(W)}:M(W) \rightarrow \Sigma M(W)$ is a split injection. Moreover, the compliment of $M(W)$ in $\Sigma M(W)$ is generated in strictly smaller degrees. Therefore, if $V$ is any finitely generated $\C$-module with $\gd(V) = d$, then $\gd(DV) \leq d-1$. In fact, it was shown by the first author and Yu that $\gd(DV) = d-1$, so long as $DV \neq 0$ \cite[Proposition 2.4]{LY}. Using the same proofs, we can actually say something a bit more general. For any positive integer $b$, Let $\tau_b$ be the map
\[
\tau_b:V \rightarrow \Sigma_b V
\]
induced by the pair $(f^b_n:\mathbf{1})$, where $f^b_n:[n] \rightarrow [n+b]$ is the standard inclusion, and $\mathbf{1}$ is the trivial map into $G$. Then $\gd(\coker(\tau_b)) < \gd(V)$.\\
\end{remark}

The derivative functor was introduced as a means of bounding the regularity of $\FI$-modules by Church and Ellenberg in \cite[Theorem A]{CE}. In that paper, Church and Ellenberg use the derivative as a convenient homological tool for approximating the homological degrees. Following this, the second author discovered that the derivative functor held much higher homological significance than was previously observed. Namely, it was shown that the derivative was critical in developing a theory of depth for $\C$-modules \cite{R}. At the same time as this, the first author and Yu also used the derivative in proving many homological facts about $\C$-modules \cite{LY}.\\

\begin{definition}\label{depth}
Let $V$ be a finitely generated $\C$-module. We define the \textbf{depth} of $V$ to be the quantity,
\[
\depth(V) := \inf\{a \mid H_1^{D^{a+1}}(V) \neq 0\} \in \mathbb{N} \cup \{\infty\},
\]
where we use the convention that the infimum of the empty set is $\infty$.\\
\end{definition}

\begin{theorem}[\cite{R}, Theorem 4.4]\label{depthclass}
Let $V$ be a finitely generated $\C$-module. Then,
\begin{enumerate}
\item $\depth(V) = 0$ if and only if $V$ is not torsion free;
\item $\depth(V) = \infty$ if and only if $V$ is $\sharp$-filtered;
\item $\depth(V) = a > 0$ is finite if and only if there is an exact sequence
\[
0 \rightarrow V \rightarrow X_{a-1} \rightarrow \ldots \rightarrow X_{0} \rightarrow V' \rightarrow 0
\]
where $X_i$ is $\sharp$-filtered for each $i$, $\gd(X_i) > \gd(X_{i-1})$, and $V'$ is not torsion free.\\
\end{enumerate}
\end{theorem}

While the above theorem justifies the use of the terminology depth, it might still be unclear at this point where this definition of depth comes from. In the remainder of this section, we will use the classical nature of the language of $k\C$-modules to provide an alternative definition of depth, which we prove is equivalent to the above.\\

\begin{definition}\label{cdepth}
Let $V$ be a finitely generated $k\C$-module, and recall that $\mi \subseteq k\C$ is the ideal generated by all non-permutations. Then we define the \textbf{classical depth} of $V$ to be the quantity
\[
\cdepth(V) := \inf\{i \mid \Ext^i_{\C\Mod}(k\C/\mi,V) \neq 0\}
\]
where we use the convention that the infimum of the empty set is $\infty$.\\
\end{definition}

\begin{remark}
As a slight technical point, one should note that the category of finitely generated $\C$-modules may not have sufficiently many injectives if $k$ is not a field of characteristic 0 (see Theorem \ref{noinj}). This is of course not problematic here, as we may compute these $\Ext$-groups through a projective resolution of $k\C/\mi$. Later, when we discuss cohomology, this will become more of an issue.\\
\end{remark}

In the paper \cite[Theorem 4.17]{R}, the second author shows that $\cdepth(V) = \depth(V)$ whenever $k$ is a field of characteristic 0. This proof, however, is highly dependent on the properties of $\C$-modules over a field of characteristic 0. The proof we provide in this section will work over any Noetherian ring.\\

We will proceed with a collection of reductions, which end with us only needing to show a particular collection of $\Ext$-groups vanish. We begin with the following lemma.\\

\begin{lemma}
Depth and classical depth are equivalent if and only if $\sharp$-filtered objects have infinite classical depth.\\
\end{lemma}

\begin{proof}
This follows from the techniques used in the proof of \cite[Theorem 4.17]{R}.\\
\end{proof}

Our second major reduction is to note that
\[
\Ext^i_{k\C}(k\C/\mi,V) \cong \prod_n \Ext^i_{k\C}(kG_n,V).
\]
It therefore suffices to show that $\Ext^i_{k\C}(kG_s,V) = 0$ for all $s$, whenever $V$ is $\sharp$-filtered. Moreover, a simple homological argument shows that it actually suffices to assume that $V = M(W)$ for some $kG_n$-module $W$.\\

\begin{lemma}
The modules $\Ext^i_{k\C}(kG_s,V)$ are zero for all $s,i \geqslant 0$ and all basic filtered modules $V$ if and only if the modules $\Ext^i_{k\C}(kG_0,V)$ are zero for all $i \geqslant 0$ and all basic filtered modules $V$. \\
\end{lemma}

\begin{proof}
The forward direction is clear. Assume that $\Ext^i_{k\C}(kG_0,V) = 0$ for all $i \geqslant 0$ and all basic filtered modules $V$.\\

A straight forward computation verifies that $L(kG_s) \cong kG_{s+1}$ for all $s \geq 0$. Therefore, if $V$ is any basic filtered module, and $i,s \geqslant 0$
\[
\Ext^i_{k\C}(kG_s,V) \cong \Ext^i_{k\C}(L_s(kG_0),V) \cong \Ext^i_{k\C}(kG_0,\Sigma_s V) = 0.
\]
Note that the second isomorphism follows from the Eckmann-Shapiro lemma and our assumption as well as Proposition \ref{genetic}.\\
\end{proof}

We have now reduced the problem enough to be solvable.\\

\begin{proof}[The proof of Theorem \ref{equivdepth}]
By the previous lemmas it suffices to prove that $\Ext^i_{k\C}(kG_0,V) = 0$, where $V = M(W)$ for some $kG_n$-module $W$.\\

We first note for any $\C$-module $V'$, elements of $\Hom_{k\C}(kG_0,V')$ can be thought of as elements of $V'_0$ which are in the kernel of all transition maps out of $V'_0$. Since basic filtered modules are torsion free, it follows that our desired result holds for $i = 0$. Next, consider the exact sequence
\begin{equation*}
0 \to J_0 \to M(0) \to kG_0 \to 0.
\end{equation*}
If $V$ is generated in degree $\geqslant 2$, then it is clear that $\Hom_{k\C}(J_0,V) = 0$. If $V$ is generated in degree 1, then an element of $\Hom_{k\C}(J_0,V)$ is a choice of an element of $V_1$ whose image under all transition maps is invariant with respect to the $G_m$-action. It is easily seen that no such elements exist in our case, and so once again $\Hom_{k\C}(J_0,V) = 0$. Finally, if $V$ is generated in degree 0, then $\Hom_{k\C}(M(0),V) \cong \Hom_{k\C}(J_0,V)$. In all cases, we conclude $\Ext^1_{k\C}(kG_0,V) = 0$.\\

Using Theorem \ref{coind}, we may write $RV \cong V \oplus V'$, where $V' = M(U)$ for some $kG_{n+1}$-module $U$.
In particular, it must be the case that $V'$ is generated in degree $\geqslant 1$, and therefore $\Hom_{k\C}(J_0, W) = 0$ by the previous paragraph's discussion. Applying the functor $\Hom_{k\C}(J_0, \bullet)$ one gets
\begin{equation*}
\Ext^1_{k\C}(J_0,V) \subseteq \Ext^1_{k\C}(J_0,R(V)).
\end{equation*}
But the Eckmann-Shapiro lemma implies
\begin{equation*}
\Ext^1_{k\C} (J_0, R(V)) \cong \Ext^1_{k\C} (\Sigma J_0, V) \cong \Ext^1_{k\C} (M(0), V) = 0.
\end{equation*}
This forces
\begin{equation*}
0 = \Ext^1_{k\C} (J_0, V) \cong \Ext^2_{k\C} (kG_0, V).
\end{equation*}
Now suppose that the conclusion holds for some $i \geqslant 2$, and consider $\Ext^{i+1}_{k\C} (kG_0, V)$. Then
\begin{equation*}
\Ext^{i+1}_{k\C} (kG_0, V) \cong \Ext^i_{k\C} (J_0, V).
\end{equation*}
Applying $\Hom_{k\C} (J_0, \bullet)$ to the exact sequence
\begin{equation*}
0 \to V \to RV \to V' \to 0
\end{equation*}
one gets
\begin{equation*}
\Ext^{i-1}_{k\C} (J_0, V') \to \Ext^i_{k\C} (J_0, V) \to \Ext^i_{k\C} (J_0, RV).
\end{equation*}
The first term is 0 by induction hypothesis, and the last term is 0 as well by using the Eckmann-Shapiro lemma, as we did previously. This proves the claim.\\
\end{proof}

As an interesting consequence of our theorem, we obtain the following vanishing theorem for $\Ext$-modules.\\

\begin{proposition}
Let $V$ be a finitely generated $\C$-module. Then for all $i \geqslant 0$, and for all $n \geqslant N(V)$,
\[
\Ext^i_{k\C}(kG_{n},V) = 0.
\]
\end{proposition}

\begin{proof}
Take $N$ to be the Nagpal number of $V$. Then for all $n \geqslant N$,
\[
\Ext^i_{k\C}(kG_n,V) \cong \Ext^i_{k\C}(L_n(kG_0),V) \cong \Ext^i_{k\C}(kG_0,\Sigma_n V) = 0
\]
since $\Sigma_n V$ is $\sharp$-filtered. This completes the proof.\\
\end{proof}

\subsection{Homological Orthogonal Relations}

In this section, we more closely examine the relationship between torsion and $\sharp$-filtered modules. In particular, over the course of this section we will be proving Theorems \ref{charfil} and \ref{homorth}.\\

\begin{definition}\label{torsion}
Let $V$ be a $\C$-module. We say that an element $v \in V_n$ is \textbf{torsion} if it is in the kernel of some transition map out of $n$. In the language of $k\C$-modules, we say that an element $v \in V_n$ is torsion, if there is some map $(f,g):[n] \rightarrow [m]$ such that $(f,g)\cdot v = 0$. We say that $V$ is \textbf{torsion} if its every element is torsion.\\

If $V$ is a $\C$-module, then there is always an exact sequence
\[
0 \rightarrow V_T \rightarrow V \rightarrow V_F \rightarrow 0
\]
where $V_T$ is a torsion module called the \textbf{torsion part} of $V$, and $V_F$ is a torsion free module called the \textbf{torsion free part} of $V$.\\
\end{definition}

While the following lemma might seem tautological, and indeed we will find it not difficult to prove, it is not immediate from the definitions thus far provided. Recall that we say that $V$ is torsion free whenever $\tau_V$ is injective.\\

\begin{lemma}
A module $V$ is torsion free if and only if it contains no torsion elements.\\
\end{lemma}

\begin{proof}
The key observation of this proof is that an element $v \in V_n$ is in the kernel of some $(f,g)_\as:V_n \rightarrow V_m$ if and only if it is in the kernel of every transition map from $n$ to $m$. Indeed, this follows from the fact that the left action of $G_m$ on $\Hom_{\C}([n],[m])$ is transitive. Moreover, we can produce, from $v$, a torsion element in $V_{m-1}$ by viewing $(f,g)$ as a composition of a map from $[n]$ to $[m-1]$ and a map from $[m-1]$ to $[m]$. It follows immediately from this that $V$ has a torsion element if and only if it has an element which is in the kernel of a transition map of the form $(f_n,\mathbf{1})$, where $f_n:[n] \rightarrow [n+1]$ is the standard inclusion and $\mathbf{1}$ is the trivial map into $G$. We recall that the map $\tau:V \rightarrow \Sigma V$ is induced by these transition maps, and therefore the proposition follows.\\
\end{proof}

\begin{lemma}\label{fintor}
Let $V$ be a finitely generated torsion module. Then $\td(V) < \infty$.\\
\end{lemma}

\begin{proof}
By assumption, there is a finite set $\{v_i\} \subseteq \sqcup V_n$ which generates $V$. By the observations in the proof of the previous lemma, for each $i$ we may find some $n_i$ such that $v_i$ is in the kernel of every transition map into $V_n$ with $n \geqslant n_i$. Because there are finitely many generators, we may find some $N$ such that all generators are in the kernel of transition maps into $V_n$ whenever $n \geqslant N$. It follows that $\td(V) \leq N$.\\
\end{proof}

We begin with the following theorem, which expands upon some of the work in previous sections.\\

\begin{theorem} \label{orthorel}
Let $V$ be a finitely generated $\C$-module. Then
\begin{enumerate}
\item $T$ is a torsion module if and only if $\Ext _{k\C}^i (T, V) = 0$ for all basic filtered modules $V$ and all $i \geqslant 0$.
\item $V$ is a $\sharp$-filtered module if and only if $\Ext _{k\C}^i (kG_s, V) = 0$ for all $s, i \geqslant 0$
\item $V$ is an injective module if and only if $\Ext _{k\C}^1 (W, V) = 0$ whenever $W$ is a basic filtered module or $W$ is a finitely generated torsion module.\\
\end{enumerate}
\end{theorem}

\begin{proof}
Let $T$ be a finitely generated torsion module. Note that $\td(T) < \infty$ by Lemma \ref{fintor}. Let $N = \td(T)$, and let $T'$ be the submodule of $T$ generated by $T_N$. Then we have an exact sequence
\[
0 \rightarrow T' \rightarrow T \rightarrow T'' \rightarrow 0
\]
where $\td(T'') < \td(T)$. Applying the functor $\Hom_{k\C}(\bullet,V)$, we find that the proposition follows if we can show $\Ext^i_{k\C}(T',V) = \Ext^i_{k\C}(T'',V) = 0$. By induction on the torsion degree, we only need to show that $\Ext^i_{k\C}(T',V) = 0$.\\

Because $T$ was finitely generated, it follows that $T'$ is finitely generated by the Noetherian property. Therefore there is an exact sequence
\[
0 \rightarrow \Omega \rightarrow kG_N^m \rightarrow T' \rightarrow 0
\]
for some module $\Omega$ which is also supported exclusively in degree $N$. Applying the functor $\Hom_{k\C}(\bullet,V)$, and using the facts that $V$ is torsion free and Theorem \ref{equivdepth}, we find that
\[
0 = \Hom_{k\C}(\Omega,V) \rightarrow \Ext^1_{k\C}(T',V) \rightarrow 0 = \Ext^1_{k\C}(kG_N^m,V).
\]
Therefore $\Ext^1_{k\C}(T',V) = 0$. Looking further along the above exact sequence, we also conclude that $\Ext^{i+1}_{k\C}(T',V) \cong \Ext^i_{k\C}(\Omega,V)$ for all $i \geqslant 1$. Using the fact that $T'$ was arbitrary, and that $\Ext^1_{k\C}(T',V) = 0$, the only if direction of the first statement now follows by induction.\\

Now we prove the if direction of the first statement; that is, $\Ext _{k\C}^i (T, V) = 0$ for all $\sharp$-filtered modules $V$ implies that $T$ is torsion. But this is clear. Indeed, if $T$ is not torsion, its torsionless part $T_F \neq 0$. Then $\Sigma_N T_F$ is a nonzero $\sharp$-filtered module for $N \gg 0$, and we get a nonzero map $T \to T_F \to \Sigma_N T_F$ where the first component is surjective and the second component is injective.\\

The second statement is simply Theorem \ref{equivdepth}, along with Theorem \ref{depthclass}.\\

One direction of the third is trivial. To prove the other direction, Theorem \ref{filcombounds} implies that for every finitely generated $k\C$-module $W$ there is a finite complex of $\sharp$-filtered modules
\begin{equation*}
0 \to W \to F^0 \to F^1 \to \ldots \to F^n \to 0.
\end{equation*}
Theorem  \ref{filcombounds} also tells us that all homologies in this complex are finitely generated torsion modules. This complex gives rise to a few short exact sequences
\begin{align*}
& 0 \to W_T \to W \to W_F \to 0,\\
& 0 \to W_F \to F^0 \to W^{(1)} \to 0,\\
& \ldots,\\
& 0 \to W^{(n)}_F \to F^n \to W^{(n+1)} \to 0,
\end{align*}
where $W^{(n+1)}$ is torsion. By assumption, one recursively deduces that $\Ext _{k\C}^i (W^{(s)}, V) = 0$ for all $i \geqslant 0$ and $0 \leqslant s \leqslant n+1$, where $W^{(0)} = W$. Therefore, $V$ is injective since $W$ is arbitrarily chosen.\\
\end{proof}

If we only consider extension groups with positive degrees, we have:\\

\begin{theorem}
Let $V$ be a finitely generated $\C$-module. Then
\begin{enumerate}
\item $\Ext _{k\C}^i (V, F) = 0$ for all basic filtered modules $F$ and all $i \geqslant 1$ if and only if $\Sigma_N V$ is a projective module for $N \gg 0$.
\item $\Ext _{k\C}^i (T, V) = 0$ for $i \geqslant 1$ and all finitely generated torsion modules $T$ if and only if $V$ is a direct sum of an injective torsion module and a $\sharp$-filtered module.\\
\end{enumerate}
\end{theorem}

\begin{proof}
If $\Sigma_N V$ is a projective module, then one has
\begin{equation*}
0 = \Ext _{k\C}^i (\Sigma_N V, F) \cong \Ext _{k\C}^i (V, R_N(F)).
\end{equation*}
But $F$ is isomorphic to a direct summand of $R_N(F)$ by Theorem \ref{coind}. Consequently, $\Ext _{k\C}^i (V, F) = 0$ for all $i \geqslant 1$ and all $\sharp$-filtered modules $F$.\\

Conversely, for $N \gg 0$, we know that $\tilde{V} = \Sigma_N V$ is a $\sharp$-filtered module. We show that if it is not projective, then there exists a $\sharp$-filtered module $F$ such that $\Ext _{k\C}^1 (V, F) \neq 0$. Suppose that $\gd(\tilde{V}) = n \geqslant 0$. Since $\tilde{V}$ is filtered, there is a short exact sequence
\begin{equation*}
0 \to V' \to \tilde{V} \to V'' \to 0
\end{equation*}
such that $V'$ is the submodule generated by $\bigoplus _{i < n} V_i$ and $V''$ is a basic filtered module generated in $n$. Without loss of generality we can assume that $V''$ is not projective since otherwise we can replace $\tilde{V}$ by $V'$ and repeat the above process.\\

Now consider $V''$ which is isomorphic to $M(V''_n)$. Since $V''$ is not projective, $V''_n$ as a $kG_n$-module cannot be projective. Therefore, we can find a finitely generated $kG_n$-module $W_n$ such that $\Ext _{kG_n}^1 (V''_n, W_n) \neq 0$. That is, there is a non-split exact sequence
\begin{equation*}
0 \to W_n \to U_n \to V''_n \to 0.
\end{equation*}
Applying the exact functor $M(\bullet)$ we get a non-split exact sequence
\begin{equation*}
0 \to M(W_n)\to M(U_n) \to M(V''_n)  \to 0.
\end{equation*}
Consequently, $\Ext _{k\C}^1 (M(V''_n), F) \neq 0$, where $F = M(W_n)$. Moreover, applying the functor $\Hom _{k\C} (\bullet, F)$ to the exact sequence
\begin{equation*}
0 \to V' \to \tilde{V} \to V'' \to 0
\end{equation*}
we deduce that $\Ext _{k\C}^1 (\tilde{V}, F) \neq 0$ since $\Ext _{k\C}^1 (V'', F) \neq 0$ and $\Hom _{k\C} (V', F) = 0$ because $\gd(V') < \gd(F) = n$. But $\tilde{V} = \Sigma_N V$. Using the adjunction, one deduces that $\Ext _{k\C}^1 (V, R_N(F)) \neq 0$. However, $R_N(F)$ is still filtered. This contradicts the given condition. In this way we prove the first statement.\\

Now we turn to the second statement. The if direction is clear. For the other direction, we consider the short exact sequence
\begin{equation*}
0 \to V_T \to V \to V_F \to 0.
\end{equation*}
Applying $\Hom_{k\C} (T, \bullet)$ and noting that $\Hom_{k\C} (T, V_F) = 0$, by the given assumption, we deduce that $\Ext _{k\C}^1 (T, V_T) = 0$. But the finitely generated torsion module $V_T$ is injective if and only if it viewed as a representation of the finite full subcategory with objects $n \leqslant \td(V_T)$ is still injective (see \cite[Section 2,4]{GL}), if and only if $\Ext _{k\C}^1 (T, V_T) = 0$ for all torsion modules. From this observation we conclude that $V_T$ is injective, so $V \cong V_T \oplus V_F$. Moreover, from the long exact sequence we conclude that $\Ext _{k\C}^1 (T, V_F) = 0$ for all $i \geqslant 0$ and all torsion modules. By the second statement of the previous theorem, $V_F$ is a $\sharp$-filtered module.\\
\end{proof}

These two results give another classification of finitely generated injective modules when $k$ is a field of characteristic 0.\\

\begin{corollary}
If $k$ is a field of characteristic 0, then every finitely generated projective module is injective as well, and every finitely generated injective module is a direct sum of a finite dimensional injective module and a projective module.\\
\end{corollary}

\begin{proof}
Let $P$ be a finitely generated projective module. We know that $\Ext _{k\C}^i (kG_s, P) = 0$ for all $i, s \geqslant 0$ since projective modules and $\sharp$-filtered modules coincide in this case. Indeed, this follows from the fact that $M(W)$ is projective if and only if $W$ is projective. For the same reason, clearly one has $\Ext _{k\C}^i (F, P) = 0$ for all finitely generated $\sharp$-filtered modules $F$. By (3) of Theorem \ref{orthorel}, $P$ is injective as well.\\

If $I$ is an injective module, then by the second statement of the previous theorem, $I$ is a direct sum of a finite dimensional injective module and a $\sharp$-filtered module, which in this case is projective.\\
\end{proof}

This fact was proven for $\FI$-modules by Sam and Snowden in \cite[Theorem 4.2.5]{SS3}. The result was generalized to $\C$-modules by the Gan and the first author in \cite[Theorem 1.7]{GL}.\\

\begin{remark}
Note that the results in this section seem to indicate that $\C$-modules are a natural candidate for a tilting theory. We do not pursue this further here, although we note that it is a possible area for further research.\\
\end{remark}

\subsection{Injective objects in the category $\C\module$}

In this section we substantiate the claim made throughout the paper that the category $\C\module$ over a Noetherian ring will often times not have sufficiently many injective objects. In fact, we will prove the stronger statement that in many cases the category does not have any torsion free injective objects. In \cite[Theorem 1.7]{GL}, as well as \cite[Theorem 4.3.1]{SS3}, It is shown that the category $\C\module$ does have sufficiently many injective objects whenever $k$ is a field of characteristic 0. The main theorem of this section therefore represents a drastic departure from the cases which were previously studied.\\

We begin with some homological lemmas, which follow from the work of the previous sections.\\

\begin{lemma} \label{injsharp}
Let $V$ be a finitely generated, torsion free injective $\C$-module. Then $V$ is $\sharp$-filtered.\\
\end{lemma}

\begin{proof}
This follows at once from Theorem \ref{homorth}.\\
\end{proof}

\begin{lemma}\label{shiftinj}
The shift functor $\Sigma$ preserves injective objects. The derivative functor $D$ preserves torsion free injective objects. That is to say, if $V$ is a torsion free injective object, then so is $DV$.\\
\end{lemma}

\begin{proof}
We have already seen that $\Sigma$ is right adjoint to the exact induction functor. This implies that it must preserve injective objects. On the other hand, if $V$ is a torsion free injective object, then there is a split exact sequence
\[
0 \rightarrow V \rightarrow \Sigma V \rightarrow DV \rightarrow 0.
\]
The fact that $\Sigma V$ is injective implies that $DV$ must be as well. Moreover, $DV$ must be torsion free, as the same is true of $\Sigma V$.\\
\end{proof}

\begin{lemma} \label{injpoint}
If $V$ is an injective $\C$-module, then $V_n$ is an injective $kG_n$-module for all $n \geq 0$.\\
\end{lemma}

\begin{proof}
Let $U \subseteq W$ be $kG_n$-modules, and assume we have a map $\phi:U \rightarrow V_n$. Then Proposition \ref{yoneda} implies that there is a map $\overline{\phi}:M(U) \rightarrow V$ such that $\overline{\phi}_n = \phi$. Because $V$ is injective, this will lift to a map $\psi:M(W) \rightarrow V$. Looking at this map in degree $n$, we obtain the desired lift of $\phi$.\\
\end{proof}

This is all we need to prove the main theorem of this section.\\

\begin{theorem} \label{noinj}
Let $k$ be a Noetherian ring, and assume that either $k$ is a field of characteristic $p > 0$, or that there are no non-trivial finitely generated injective $k$-modules. Then the category $\C\module$ does not admit any torsion free injective modules over $k$. In particular, under either of the above hypotheses, the category $\C\module$ does not have sufficiently many injective objects.\\
\end{theorem}

\begin{proof}
Let $V$ be a finitely generated, torsion free injective $\C$-module. Then $V$ is $\sharp$-filtered by Lemma \ref{injsharp}. Moreover, \cite[Proposition 2.4]{LY} tells us that $\gd(DV) = \gd(V) - 1$, so long as $\gd(V) \neq 0$. In particular, Lemma \ref{shiftinj} implies that if we apply the derivative functor enough times to $V$, we will be left with a torsion free, injective module, which is generated in degree $0$. All such modules take the form $M(W)$, where $W$ is a finitely generated module over $kG_0 \cong k$. Note that Lemma \ref{injpoint} implies that $W$ is injective as a $k$-module.\\

Now assume that $k$ is a field of characteristic $p > 0$. Then $M(0)$ will be a summand of $M(W)$, and therefore will be injective as well. This is a contradiction of Lemma \ref{injpoint}, as $M(0)_n$ is the trivial module for all $n$, and the trivial module is not injective for $n \geqslant p$.\\

If, on the other hand, $k$ satisfies the second of our two conditions, then we reach a contradiction with the fact that $W$ must be a finitely generated injective module.\\
\end{proof}

\begin{remark}
The above theorem was proven by Changchang Xi and the first author during the latter's visit to Capital Normal University in December of 2015. It was independently proven by the second author a short time later. The first author would like to thank Prof. Xi for hosting him during this visit, and both authors thank Prof. Xi for kindly allowing us to include this result in the paper.\\
\end{remark}

\begin{remark}
In \cite[Theorem 2.5.1]{SS3}, Sam and Snowden prove that the category $\C\module^{tor}$ and the Serre quotient category $\C\module/\C\module^{tor}$ are equivalent whenever $k$ is a field of characteristic 0. In \cite{GL}, it is shown that $\C\module^{tor}$ has sufficiently many injective objects whenever $k$ is a field. The above theorem therefore seems to indicate that the equivalence of Sam and Snowden will fail whenever $k$ has positive characteristic.\\
\end{remark}

\section{Local Cohomology}

In this final portion of the paper, we aim to develop a theory of local cohomology for finitely generated $\C$-modules. This problem was first considered by Sam and Snowden in \cite{SS3}. Their work only applies to the case where $k$ is a field of characteristic 0. Much of the difficulty of treating the theory over a general Noetherian ring is that it is unclear whether the category $\C\module$ has sufficiently many injective objects. Despite this fact, we will discover that $\sharp$-filtered objects can play the role of injective modules in many computations. The reader is encouraged to compare the theorems in this part of the paper to theorems in the local cohomology of modules over a polynomial ring.\\

\subsection{The Torsion Functor}

\begin{definition}
We write $\C\Mod^{tor}$ for the category of torsion $\C$-modules. We will also used $\C\module^{tor}$ to denote the category of finitely generated torsion modules. By Lemma \ref{fintor}, $\C\module^{tor}$ is equivalent to the category of finitely generated $\C$-modules with finite support.\\

The \textbf{torsion functor} $\widetilde{H_\mi^0}:\C\Mod \rightarrow \C\Mod^{tor}$ is defined by setting $H^0_\mi(V)$ to be the maximal torsion submodule of $V$. We will use $H^0_\mi: \C\module \rightarrow \C\module^{tor}$ to denote the restriction of $\widetilde{H_\mi^0}$ to $\C\mod$.\\
\end{definition}

Our goal for the remainder of this paper is to study the torsion functor and its derived functors. Unfortunately, it is not yet clear that this functor actually has derived functors. When $k$ is a field of characteristic 0 Sam and Snowden \cite[Theorem 4.3.1]{SS3}, as well as Gan and the first author \cite[Theorem 1.7]{GL}, have shown that the category $\C\module$ has sufficiently many injective objects. When $k$ is not a field of characteristic 0 this is no longer the case. Luckily, we can easily show that the larger category $\C\Mod$ does have sufficiently many injective objects.\\

\begin{proposition}
The category $\C\Mod$ has sufficiently many injective objects.\\
\end{proposition}

\begin{proof}
It suffices to show that $\C\Mod$ satisfies Grothendieck's AB Criterion \cite[Theorem 1.10.1]{G}. This follows from the fact that it is a functor category from a small category ($\C$) into a category which satisfies this criterion ($k\Mod$).\\
\end{proof}

Because of this proposition, we know that we may at least make sense of the derived functors of $\widetilde{H^0_\mi}$.\\

\begin{definition}
We write $\widetilde{H^i_\mi}$ to denote the $i$-th right derived functor of $\widetilde{H^0_\mi}$. We refer to these as the \textbf{local cohomology functors}.\\
\end{definition}

Our goal for much of what follows will be to show that if $V$ is a finitely generated $\C$-module, then the modules $\widetilde{H^i_\mi}(V)$ are also finitely generated. Once this is accomplished, we will spend the remainder of the paper showing how this result applies to invariants such as the Nagpal number and regularity.\\

\subsection{Some Acyclics}

In this section we will classify two important families of finitely generated modules which are acyclic with respect to local cohomology: $\sharp$-filtered objects and torsion modules. To accomplish this, we will need to view the torsion functor from a slightly different perspective.\\

\begin{definition}
For each $n \geqslant 1$, and each $r \geqslant 0$, we define the $\C$-module $M(r)/\mi^n$ as the quotient of $M(r)$ by the submodule generated by $M(r)_{r+n}$. We also use $\fiHom(k\C/\mi^n,\bullet):\C\Mod \rightarrow \C\Mod$ to denote the functor
\[
\fiHom(k\C/\mi^n,V) := \bigoplus_r \Hom_{k\C}(M(r)/\mi^n, V).
\]
We set $\fiExt^i(k\C/\mi^n,\bullet):\C\Mod \rightarrow \C\Mod$ to be the $i$-th derived functor of $\fiHom(k\C/\mi^n,\bullet)$. Explicitly,
\[
\fiExt^i(k\C/\mi^n,V) := \bigoplus_r \Ext^i_{k\C}(M(r)/\mi^n, V)
\]
\end{definition}

\begin{remark}
We note that the $\C$-modules $\Hom_{k\C}(k\C/\mi^n, V)$ and $\fiHom(k\C/\mi^n,V)$ are not necessarily isomorphic. Indeed, morphisms which appear as elements in the latter module are necessarily supported in finitely many degrees, while this is not the case for the prior module.\\

Also note that if we restrict ourselves to working with finitely generated modules, then $\fiHom(k\C/\mi^n,\bullet)$ and $\Hom_{k\C}(k\C/\mi^n,\bullet)$ are isomorphic. This follows from Lemma \ref{fancyhom}, as well as the fact that the torsion parts of finitely generated modules are finitely supported.\\
\end{remark}

\begin{lemma}\label{fancyhom}
Let $r \geqslant 0$ and $n \geqslant 1$ be integers, and let $V$ be a $\C$-module. Then $\fiHom(k\C/\mi^n,V)$ is naturally isomorphic to a submodule of $V$. Namely,
\begin{eqnarray*}
\fiHom(k\C/\mi^n,V)_r &=\text{ those elements of $V_r$ which are in the kernel of every transition map}\\
												   &\text{ $(f,g):[r] \rightarrow [s]$ with $s-r \geqslant n$.}
\end{eqnarray*}
\end{lemma}

\begin{proof}
By definition, $\fiHom(k\C/\mi^n,V)_r = \Hom_{k\C}(M(r)/\mi^n,V)$. Proposition \ref{yoneda} implies that any morphism $M(r) \rightarrow V$ is determined by the image of the identity in degree $r$. It follows that a morphism $M(r)/\mi^n \rightarrow V$ is determined by a choice of element in $V_r$, with the added restriction that it is in the kernel of all transition maps into $V_{r+n}$. This completes the proof.\\
\end{proof}

\begin{proposition}\label{homversion}
There is a natural isomorphism of functors,
\[
\widetilde{H^0_\mi}(\bullet) \cong \lim_{\rightarrow} \fiHom(k\C/\mi^n,\bullet).
\]
More generally, there are natural isomorphisms
\[
\widetilde{H^i_\mi}(\bullet) \cong \lim_{\rightarrow} \fiExt^i(k\C/\mi^n,\bullet)
\]
for all $i \geqslant 0$.\\
\end{proposition}

\begin{proof}
The fact that $\lim_{\rightarrow} \fiExt^i(k\C/\mi^n,\bullet)$ are the derived functors of $\lim_{\rightarrow} \fiHom(k\C/\mi^n,\bullet)$ follows from the fact that filtered colimits are exact, as well as the relevant definitions. It therefore suffices to prove the first statement. This statement follows immediately from the previous lemma.\\
\end{proof}

This new perspective on the torsion functor will allow us to use what we already proved about $\sharp$-filtered modules to conclude that they are acyclic with respect to local cohomology.\\

\begin{corollary}\label{filacyclic}
If $V$ is a $\sharp$-filtered $\C$-module, then $\widetilde{H^i_\mi}(V) = 0$ for all $i$.\\
\end{corollary}

\begin{proof}
Follows immediately from the previous proposition, as well as the first part of Theorem \ref{orthorel}.\\
\end{proof}

Proposition \ref{homversion} also implies that finitely generated torsion modules are acyclic with respect to $\widetilde{H^0_\mi}$. Indeed, while it is not the case that $\fiExt^i(k\C/\mi^n,V) = 0$ for all $n$ and all torsion modules $V$, this is the case for $n$ sufficiently large.\\

\begin{corollary} \label{toracyclic}
Let $V$ be a finitely generated, torsion $\C$-module. Then for all $n \gg 0$ and all $i \geqslant 1$, $\fiExt^i(k\C/\mi^n, V) = 0$. In particular, $\widetilde{H^i_\mi}(V) = 0$ for $i \geqslant 1$.\\
\end{corollary}

\begin{proof}
Fix some integer $r \geqslant 0$ and $i \geqslant 1$. It suffices to show that $\Ext^i_{k\C}(M(r)/\mi^n,V) = 0$ for all $n \gg 0$. Write $K^{(r,n)}$ for the submodule of $M(r)$ generated by $M(r)_{n+r}$. Then by definition there is an exact sequence,
\[
0 \rightarrow K^{(r,n)} \rightarrow M(r) \rightarrow M(r)/\mi^n \rightarrow 0.
\]
Using the fact that $M(r)$ is projective, we conclude that for all $i \geqslant 1$ there is an exact sequence
\[
\Ext^{i-1}_{k\C}(K^{(r,n)},V) \rightarrow \Ext^i_{k\C}(M(r)/\mi^n,V) \rightarrow 0.
\]
Choose $n$ such that $r + n > \td(V)$. We may construct a projective resolution of $K^{(r,n)}$, say $F^\bullet \rightarrow  K^{(r,n)} \rightarrow 0$, such that for all $j$, and all $m < r+n$, $F^j_{m} = 0$. The module $\Ext^{i-1}_{k\C}(K^{(r,n)},V)$ is a subquotient of the module $\Hom_{k\C}(F^{i-1},V)$, which is zero by Proposition \ref{yoneda} and our choice of $n$. This completes the proof of the first statement. The second statement is an immediate consequence of Proposition \ref{homversion}.\\
\end{proof}

\subsection{Computing Local Cohomology}

In this section we present a complex $\mathcal{C}^\bullet V$, associated to a finitely generated $\C$-module $V$, whose cohomology modules are precisely the local cohomology modules of $V$. This complex will allow us to show that the local cohomology modules of $V$ are always finitely generated, and will allow us to relate local cohomology to the Nagpal number and the regularity of $V$.\\

Let $V$ be a finitely generated $\C$-module. By Nagpal's Theorem, we may find an integer $b_{-1}$ such that $\Sigma_{b_{-1}}V$ is $\sharp$-filtered. This yields an exact sequence
\[
V \stackrel{\tau_{b_{-1}}}\rightarrow \Sigma_{b_{-1}}V \rightarrow Q^{(0)} \rightarrow 0
\]
where $\tau_{b_{-1}}$ is the map defined in remark \ref{dervprop}. Call $F^0 := \Sigma_{b_{-1}}V$, and recall from Remark \ref{dervprop} that $\gd(Q^{(0)}) < \gd(V)$. Proceeding inductively, we may find an integer $b_i$ for which $F^{i+1} := \Sigma_{b_i}Q^{(i)}$ is $\sharp$-filtered. We also have maps $\partial^i:F^i \rightarrow F^{i+1}$ defined by the composition of the quotient map $F^i \rightarrow Q^{(i)}$, and the map $\tau_{b_{i+1}}:Q^{(i)} \rightarrow F^{i+1}$. Putting it all together, we obtain a complex
\[
\mathcal{C}^\bullet V : 0 \rightarrow V \rightarrow F^{0} \rightarrow \ldots \rightarrow F^{n} \rightarrow 0
\]
Note that this complex is necessarily bounded by our observation that the generating degree of $Q^{(i)}$ is always strictly less than that of $Q^{(i-1)}$. If we define $Q^{(-1)} := V$, then one also notes that
\[
H^{i}(\mathcal{C}^\bullet V) = \widetilde{H^0_\mi}(Q^{i})
\]
The complex $\mathcal{C}^\bullet V$ was first introduced by Nagpal in \cite[Theorem A]{N} and was rediscovered by the first author and Yu in \cite[Theorem 1.7]{LY}. Following this, the first author used this complex to prove bounds on the regularity and the Nagpal number of $V$, which we saw in Theorem \ref{filcombounds}.\\

\begin{remark}
One may have noted that the construction of the complex $\mathcal{C}^\bullet V$ depended on the integers $b_i$. Indeed, the assignment $V \mapsto \mathcal{C}^\bullet V$ is non-functoral in the category of chain complexes of $\C$-modules. However, this assignment is functoral in the derived category. That is to say, choosing different values for the integers $b_i$ yields a complex which is quasi-isomorphic to the original complex.\\
\end{remark}

One thing that is important for the present work, is that this complex actually computes the local cohomology modules of $V$.\\

\begin{proof}[Proof of Theorem \ref{coequiv}]
Recall the modules $Q^{(i)}$ defined during the construction of $\filcom V$. We have already noted that $H^i(\filcom V) \cong \widetilde{H^0_\mi}(Q^{(i)})$. We will prove that
\[
\widetilde{H^0_\mi}(Q^{(i)}) \cong \widetilde{H^{i+1}_\mi}(V).
\]
The claim is clear when $i = -1$. Otherwise, we have an exact sequence
\[
0 \rightarrow Q^{(i)}_T \rightarrow Q^{(i)} \rightarrow Q^{(i)}_F \rightarrow 0
\]
where $Q^{(i)}_T$ is the torsion part of $Q^{(i)}$ and $Q^{(i)}_F$ is the torsion free part. Corollary \ref{toracyclic} implies $\widetilde{H^i_\mi}(Q^{(i)}_F) \cong \widetilde{H^{i}_\mi}(Q^{(i)})$ for $i \geqslant 1$. Next, we look at the exact sequence
\[
0 \rightarrow Q^{(i)}_F \rightarrow F^{i+1} \rightarrow Q^{(i+1)} \rightarrow 0
\]
and apply Corollary \ref{filacyclic} to conclude
\[
\widetilde{H^0_\mi}(Q^{(i+1)}) \cong \widetilde{H^{1}_\mi}(Q^{(i)}_F) \cong \widetilde{H^{1}_\mi}(Q^{(i)}) .
\]
We reach our desired conclusion by induction.\\

The bounds given in the theorem follow immediately from Theorem \ref{filcombounds}.\\
\end{proof}

We are now free for the remainder of the paper to consider all local cohomology modules as existing inside the category $\C\module^{tor}$. In particular, it is no longer necessarily to distinguish $H^0_\mi$ from $\widetilde{H^0_\mi}$.\\

\begin{definition}
We write $H^i_\mi:\C\module \rightarrow \C\module^{tor}$ to denote the $i$-th derived functor of $H^0_\mi$.\\
\end{definition}

\subsection{Applications of Theorem \ref{coequiv}}

We spend this section exploring the plethora of applications of Theorem \ref{coequiv}. In particular, we will prove Theorem \ref{localcohomologybounds}, along with many other results. To begin, we obtain a new homological characterization of $\sharp$-filtered modules.\\

\begin{proposition}\label{exactfil}
Let $V$ be a finitely generated $\C$-module. Then $H^i_\mi(V) = 0$ for all $i \geqslant 0$ if and only if $V$ is $\sharp$-filtered. In particular, $V$ is $\sharp$-filtered if and only if $\filcom V$ is exact.\\
\end{proposition}

\begin{proof}
The backwards direction follows from Corollary \ref{filacyclic}. If $H^i_\mi(V) = 0$ for all $i$, then $\filcom V$ is exact in all degrees by Theorem \ref{coequiv}. This implies $V$ is $\sharp$-filtered by Theorem \ref{homacyclic}.\\
\end{proof}

This will allow us to classify all modules which are acyclic with respect to the torsion functor.\\

\begin{proposition}
Let $V$ be a finitely generated $\C$-module. Then $V$ is acyclic with respect to the torsion functor if and only if its torsion free part $V_F$ is $\sharp$-filtered.\\
\end{proposition}

\begin{proof}
The exact sequence
\[
0 \rightarrow V_T \rightarrow V \rightarrow V_{F} \rightarrow 0
\]
implies that $H^i_\mi(V) \cong H^i_\mi(V_F)$ for all $i \geqslant 1$ by Corollary \ref{toracyclic}. The previous proposition now implies our result.\\
\end{proof}

The fact that the complex $\filcom V$ is bounded from above implies that $H^i_\mi(V) = 0$ for all $i \gg 0$. We may therefore make sense of the following definition.\\

\begin{definition}
Let $V$ be a finitely generated $\C$-module which is not $\sharp$-filtered. Then we define its \textbf{cohomological dimension} to be the quantity
\[
\cd(V) := \sup\{i \mid H^i_\mi(V) \neq 0\} \in \mathbb{N}.
\]
\end{definition}

\begin{proposition}
If $V$ is a finitely generated $\C$-module which is not $\sharp$-filtered, Then any non-trivial local cohomology modules $H^i_\mi(V)$ must have
\[
\depth(V) \leq i \leq \cd(V).
\]
Moreover, $H^i_\mi(V) \neq 0$ at each of the two extremes.\\
\end{proposition}

\begin{proof}
We only need to show that $H^i_\mi(V) = 0$ for all $i < \depth(V)$, and that $H^i_\mi(V) \neq 0$ for $i = \depth(V)$. If $\delta$ is the smallest value for which $H^i_\mi(V) \neq 0$, then Theorem \ref{coequiv} implies that there is an exact sequence,
\[
0 \rightarrow V \rightarrow F^0 \rightarrow \ldots \rightarrow F^{\delta-1} \rightarrow V' \rightarrow 0
\]
where $F^i$ is $\sharp$-filtered, and $V'$ has torsion. This implies that $\depth(V) = \delta$ by Theorem \ref{depthclass}.\\
\end{proof}

\begin{corollary}
Let $V$ be a finitely generated $\C$-module which is not $\sharp$-filtered. Then $\cd(V) \leq \gd(V)$.\\
\end{corollary}

We next turn our attention to the relationship between local cohomology and the Nagpal number and regularity. We begin with the following observation.\\

\begin{theorem}
Let $V$ be a finitely generated $\C$-module which is not $\sharp$-filtered. Then
\[
N(V) = \max\{\td(H^i_\mi(V))\mid i \geqslant 0\} + 1.
\]
In particular,
\[
N(V) \leq \max\{\td(V),2\gd(V) - 2\} + 1.
\]
\end{theorem}

\begin{proof}
Proposition \ref{exactfil} tells us that $V$ is $\sharp$-filtered if and only if $\filcom V$ is exact. Using the fact that the shift functor is exact, and the construction of the complex $\filcom V$, it follows that $\Sigma_b V$ is filtered whenever
\[
b > \max\{\td(H^i(\filcom V)) \mid i \geqslant -1\} = \max\{\td(H^i_\mi(V))\mid i \geqslant 0\}.
\]
The desired bound follows immediately from Theorem \ref{filcombounds}.\\
\end{proof}

Note that this bound on the Nagpal number is not new. Indeed, it was proven by the first author in \cite[Theorem 1.3]{L2}. A similar bound was found by the second author in \cite[Theorem D]{R}.\\

\begin{theorem}\label{regbound}
Let $V$ be a finitely generated $\C$-module. Then,
\[
\reg(V) \leqslant \max\{\td(H^i_\mi(V)) + i\}.
\]
In particular,
\[
\reg(V) \leqslant \max\{2\gd(V ) - 1, \td(V )\}.
\]
\end{theorem}

\begin{proof}
We proceed by induction on the generating degree of $V$. The bound is vacuously true if $V = 0$. Assume that $V \neq 0$, and note we have an exact sequence
\[
0 \rightarrow V_T \rightarrow V \rightarrow V_F \rightarrow 0.
\]
Applying Corollary \ref{toracyclic} we find that $H^i_\mi(V) \cong H^i_\mi(V_F)$ for all $i \geqslant 1$. Because $V_F$ is torsion free, Remark \ref{dervprop} implies we have an exact sequence
\[
0 \rightarrow V_F \rightarrow F \rightarrow C \rightarrow 0
\]
where $F$ is $\sharp$-filtered, and $C$ is generated in strictly smaller degree. Applying the homology functor, and using Theorem \ref{homacyclic} we find by induction that for all $i \geqslant 1$
\[
\td(H_i(V_F)) - i = \td(H_{i+1}(C))- (i+1) +1 \leq \reg(C) + 1 \leq \max\{\td(H^s_\mi(C)) + s \mid s \geqslant 0\} + 1.
\]
If we instead apply the torsion functor to this exact sequence, we find that $H^s_\mi(C) \cong H^{s+1}_\mi(V)$ for all $s \geqslant 0$. Therefore, recalling that regularity only requires we bound the higher homologies (Definition \ref{regdef}),
\[
\reg(V_F) \leq \max\{\td(H^i_\mi(V)) + i \mid i > 0\}.
\]
On the other hand, $V_T$ is a torsion module and therefore \cite[Theorem 1.5]{L} \cite[Corollary 3.11]{R} imply that $\reg(V_T) \leq \td(V_T) = \td(V)$. Putting everything together we obtain our desired bound.\\

The second bound follows immediately from the first and Theorem \ref{filcombounds}.\\
\end{proof}

As with the bounds on the Nagpal number, the second bound in the above theorem is not new. The first author had discovered this bound earlier in \cite[Theorem 1.3]{L2}.\\

\subsection{A Conjecture and its Consequences}

In this section we state our primary conjecture, which firmly establishes the relationship between local cohomology and regularity. Following this, we take time to illustrate some interesting consequences of the conjecture.\\

\begin{conjecture}\label{mainconj}
Let $V$ be a finitely generated $\C$-module which is not $\sharp$-filtered. Then,
\[
\reg(V) = \max\{\td(H^i_\mi(V)) + i\}.
\]
\end{conjecture}

We have already seen that $\max\{\td(H^i_\mi(V)) + i\}$ is an upper bound on the regularity of $V$ (Theorem \ref{regbound}). The opposite inequality seems to be much harder to prove. The reader familiar with classical local cohomology theory may recognize a similar statement from local cohomology of modules over a polynomial ring \cite[Theorem 4.3]{E}. What is interesting is that the proofs that $\reg(V) \geqslant \max\{\td(H^i_\mi(V)) + i\}$ in that context often proceed by induction on the projective dimension. Theorem \ref{homacyclic} suggests that such an approach will not work.\\

We now spend the remainder of this section detailing some corollaries to the above.\\

\begin{corollary}
Let $V$ be a finitely generated torsion module, and assume Conjecture \ref{mainconj}. Then
\[
\reg(V) = \td(V)
\]
\end{corollary}

\begin{proof}
This follows immediately from Conjecture \ref{mainconj} and Corollary \ref{toracyclic}.\\
\end{proof}

What is perhaps more interesting, is what the conjecture implies about the relationship between regularity and the shift functor. To see this, first note that for all $i \geqslant 0$, $H^i_\mi(\Sigma V) \cong \Sigma H^i_\mi(V)$. Indeed, this follows from Theorem \ref{coequiv}, as well as how the complex $\filcom V$ was constructed. We therefore conclude the following.\\

\begin{corollary}
Let $V$ be a finitely generated $\C$-module such that $\Sigma V$ is not $\sharp$-filtered, and assume Conjecture \ref{mainconj}. Then
\[
\reg(\Sigma V) = \reg(V) - 1
\]
\end{corollary}

\begin{proof}
Assuming the conjecture we have,
\[
\reg(\Sigma V) = \max\{\td(H^i_\mi(\Sigma V)) + i\} = \max\{\td(\Sigma H^i_\mi(V)) + i\} = \max\{\td(H^i_\mi(V)) + i\} -1 = \reg(V) - 1.
\]
\end{proof}

\end{document}